\documentclass[reqno]{amsart}
\usepackage{pdfsync}

\usepackage{amsmath}%
\usepackage{amsthm}%
\usepackage{amsfonts}
\usepackage{amssymb}%
\usepackage{graphicx}
\usepackage{color}
\usepackage{subfigure}
\newtheorem{thm}{Theorem}[section]
\theoremstyle{plain}

\newtheorem{lemma}[thm]{Lemma}
\newtheorem{proposition}[thm]{Proposition}
\theoremstyle{remark}

\newtheorem{remark}[thm]{Remark}

\newcommand{\OS}[2]{OP(Lap = #1, Bdry = #2)}
\newcommand{\R}{{\rm I \! R}}
\newcommand{\chisub}[1]{\chi_{_{#1}}}
\newcommand{\Cartman}{\Omega}

\numberwithin{equation}{section}

\begin{document}
\title[Perturbed Obstacle Problems]
{Perturbed Obstacle Problems in Lipschitz Domains: Linear Stability and Non--Degeneracy in Measure}

\author[I. Blank]{Ivan Blank}
\address{Department of Mathematics \\ Kansas State University \\ Manhattan, KS USA}
\email{blanki@math.ksu.edu}%
\urladdr{http://www.math.ksu.edu/~blanki}

\author[J. LeCrone]{Jeremy LeCrone}
\address{Department of Mathematics \& Computer Science \\ University of Richmond\\ Richmond, VA USA}
\email{jlecrone@richmond.edu}%
\urladdr{http://math.richmond.edu/faculty/jlecrone/}





\begin{abstract}
We consider the classical obstacle problem 
on bounded, connected Lipschitz domains $D \subset \R^n$. 
We derive quantitative bounds on the changes to contact sets under general 
perturbations to both the right hand side and the boundary data for obstacle problems. 
In particular, we show that the Lebesgue measure of the symmetric difference 
between two contact sets is linearly comparable to the $L^1$--norm
of perturbations in the data.
\end{abstract}

\maketitle

%
%
%
%
%


%
%

\section{Introduction}

Given functions $g_1, g_2 : D \to [\lambda, \mu]$ and 
$\psi_1, \psi_2 : \partial D \to [0,\infty),$ with sufficient regularity and 
$0 < \lambda \le \mu$, we denote by $\OS{g_i}{\psi_i}$ the non--negative 
functions $u_i \in W^{1,2}(D)$ satisfying the semilinear pdes
\begin{equation}\label{ObsProb}
	\begin{cases}
		\Delta u_i = \chisub{ \{u_i > 0 \} } g_i & \text{in $D$,}\\
		u_i = \psi_i & \text{on $\partial D$},
	\end{cases}
	\qquad i = 1,2.
\end{equation}
We mention that the obstacle problem can also be formulated in terms of
variational inequalities and functional optimization, though the equivalence 
of these settings is well--known (c.f. \cite{F,R}, for instance).
The existence and uniqueness of solutions to \eqref{ObsProb} is also
shown in \cite{F,R}, via standard methods in functional analysis.

Under minimal assumptions on the data $(g_i, \psi_i)$ and the content of contact sets
$\Lambda(u_i) := \{ x \in B : u_i(x) = 0 \}$, we prove that the Lebesgue measure 
of the symmetric difference $\Lambda(u_1) \ \Delta \ \Lambda(u_2)$
is linearly comparable to the $L^1$--norms of the perturbations to data 
over appropriate sets. This result is stated in the following theorem:


\begin{thm}\label{MainResult}
	Let $D \subset \R^n$ be a bounded, connected Lipschitz domain and let
	\[
		(g_i, \psi_i) \in L^{\infty}(D) \times C(\partial D), \qquad i = 1,2,
	\]
	with $0 < \lambda \le g_i \le \mu$ and $\psi_i \ge 0$. 
	Consider the following obstacle problem solutions
	\begin{equation}\label{Solutions}
	\begin{aligned}
		u_i &= \OS{g_i}{\psi_i}, \qquad \text{and}\\
		\bar v &= \OS{\min (g_1, g_2)}{\max (\psi_1, \psi_2)}.
	\end{aligned}
	\end{equation}
	Assume there exist	$\bar{y} \in D$ and $\delta > 0$ so that 
	$B_{\delta}(\bar y) \subset \Lambda(\bar v) := \{ \bar v = 0 \}.$
	Further, for $\eta > 0$, define the set 
	\[
		D_{-\eta} := D \setminus \mathcal{N}_{\eta}(\partial D) = 
			D \setminus \{ x \in \R^n : {\rm dist}(x, \partial D) < \eta \},
	\] 
	then:
	\begin{enumerate}
		\item[(a)] \textit{(Linear Stability)} There exist positive constants
			$C_1$ and $C_2$ so that
			\begin{equation}\label{MeasStability}
				|(\Lambda(u_1) \; \Delta \; \Lambda(u_2)) \cap D_{-\eta}|
					\le C_1 \|\psi_1 - \psi_2 \|_{L^1(\partial D)} + 
					C_2 \| g_1 - g_2 \|_{L^1(\Omega(\bar v))} .
			\end{equation}
		\item[(b)] \textit{(Linear Non--Degeneracy)} 
			If $\psi_1 \ge \psi_2$ on $\partial D$
			and $g_1 \le g_2$ in $D$, then there exist positive constants
			$C_3$ and  $C_4$ so that
			\begin{equation}\label{MeasNonDegen}
				|\Lambda(u_1) \; \Delta \; \Lambda(u_2)| \ge 
					C_3 \| \psi_1 - \psi_2 \|_{L^1(\partial D)} + 
					C_4 \| g_1 - g_2 \|_{L^1(\Omega(u_1) \cap D_{-\eta})} ,
			\end{equation}
			where $\Omega(u_i) := \{ u_i > 0 \}$ denotes the non--contact
			set for $u_i.$
	\end{enumerate}
\end{thm}

\begin{remark}
	{\bf (a)}
	Use of the term \textbf{non--degeneracy} here differs 
	from most literature related to the obstacle problem. 
	Typically, one refers to the non--degenerate quadratic growth enjoyed by 
	solutions to the obstacle problem in non--contact regions, 
	while here we refer to the non--degenerate changes to contact regions
	induced by data perturbations.
	
	\medskip
	{\bf (b)} 
	We will also make use of the function 
	\[
		\underline{v} := \OS{\max(g_1,g_2)}{\min(\psi_1,\psi_2)}
	\]
	in the proof of Theorem~\ref{MainResult}. By maximum principle, 
	one immediately concludes $\underline{v} \le u_i \le \overline{v}$, $i = 1,2$.
\end{remark}

Comparing these results with the literature, a form of measure stability is
proved in \cite{C3}, with square root dependence on changes to the data,
while many more stability results appear in \cite{R}, 
including stability with respect to perturbations to
the operator itself, which we do not treat here.
On the other hand, all of the quantitative bounds established in \cite{R} 
also involve the square root of data perturbations (along with many convergence
results without giving a rate).
The closest result to our current linear stability (Theorem~\ref{MainResult}(a))
can be found in \cite[Theorem 4.1]{B}, where the first author 
worked in the specific setting of $D = B_1$, the unit ball in $\R^n.$
We note that the result in \cite{B}
measures the full set $\Lambda(u_1) \ \Delta \ \Lambda(u_2)$, while the current
work measures only the portion of this symmetric difference that is away from
the boundary $\partial D$ by some distance $\eta > 0$, however we are
working in a more general setting here and consider both perturbations to the 
right hand side and boundary data for obstacle problems.

Regarding linear non--degeneracy, our result (Theorem~\ref{MainResult}(b))
appears to be new in the literature.
One can find a form of linear non--degeneracy bounds in \cite[Theorem 5.7]{B},
where it is established that the Hausdorff distance between free boundaries is linearly
comparable to perturbations of the Laplacian data, in the special case when 
free boundaries are assumed to be regular. 
We note that the current work differs from \cite[Theorem 5.7]{B} as we 
do not assume any regularity on the free boundaries, we permit perturbations
to the Laplacian that are supported on proper subsets of the domain $B$ (whereas
the argument in \cite{B} requires the difference $g_2 - g_1$ to be uniformly 
bounded below by some positive constant), and we allow perturbations
to both the right hand side and boundary data. 

As a final note on literature related to perturbed obstacle problems, the reader 
should refer to \cite{SS} for precise formulas for normal velocity
and acceleration of free boundaries under sufficiently regular variations to
Laplacian and boundary data. The authors of \cite{SS} work 
in a global setting (i.e. $D = \R^n$) with compactly supported perturbations 
to Laplacian data and constant ``boundary'' data (at $|x| \to \infty$).
Finally, we note that regularity of free boundaries is assumed in \cite{SS},
as one requires to make sense of pointwise normal velocity.

Outlining the current work, in Section~\ref{sec:Setting} we introduce 
notation and state necessary lemmas from
elliptic theory and potential theory.
Then, in Section~\ref{sec:Proof}, we prove Theorem~\ref{MainResult} by 
splitting into cases where either boundary data or Laplacian data are fixed.

\section{Setting, Notation, and Preliminary Bounds}\label{sec:Setting}

We assume the set $D \subset \R^n$ is a bounded, connected Lipschitz domain. 
In this section, we collect preliminary lemmas 
we will use in the proof of Theorem~\ref{MainResult}.

\subsection{The Inhomogeneous Dirichlet Problem in $D$}

Considering the situation in \eqref{ObsProb} when boundary data is fixed 
(i.e. assuming $\psi_1 = \psi_2$), 
the difference $w = u_1 - u_2$ will satisfy an inhomogeneous Dirichlet problem of the form
\begin{equation}\label{InHomoEqn}
	\begin{cases}
		\Delta w = f & \text{in $D$}\\
		w = 0 & \text{on $\partial D$.}
	\end{cases}
\end{equation}
The precise expression of the function $f$ is not important at the moment (though it may be
instructive for the reader to identify values of $f$ on subsets of $D$ depending upon the 
contact sets $\Lambda(u_1)$, $\Lambda(u_2)$, and regions of overlap between these), 
rather we note that tools for controlling solutions to \eqref{InHomoEqn} with rough data
$f$ will thus help control differences between $u_1$ and $u_2$. We direct the reader to
\cite{JK} for a detailed treatment of inhomogeneous Dirichlet problems in Lipschitz domains,
though many of the statements below come from \cite{Sh}.

We first note that \eqref{InHomoEqn} is solvable for general domains $\Omega$ and data $f$:

\begin{thm}\cite[Theorem 1.2.1]{Sh}:
Let $\Omega$ be a bounded domain in $\R^n$. 
Given any $f \in W^{-1}(\Omega)$ (the dual space to $W^{1,2}_0(\Omega)$), there exists a unique
solution $u = Tf \in W^{1,2}_0(\Omega)$ to \eqref{InHomoEqn}, in the sense that
\[
	\int_\Omega \nabla u \, \nabla v = \int_\Omega f v \qquad \text{for all $v \in W^{1,2}_0(\Omega)$.}
\]	
\label{Solvable}
\end{thm}

We note that there exists a {\it Dirichlet Green's function} for any bounded $\Omega$:

\begin{thm}\cite[Theorem 1.2.2]{Sh}: Let $\Omega$ be a bounded domain in $\R^n$ and 
let $T: W^{-1}(\Omega) \to W^{1,2}_0(\Omega)$ be the operator defined in Theorem~\ref{Solvable}.
There exists a kernel function $G(x,y)$ in $\Omega \times \Omega$ satisfying the following:
\begin{enumerate}
	\item[(a)] $G(x,y) \in C^{\infty}(\Omega \times \Omega \setminus \{(x,x) : x \in \Omega\} )$
	\item[(b)] $(1 - \eta_y(x)) G(x,y) \in W^{1,2}_0(\Omega)$ where $\eta_y(x) \in C^\infty_0(\Omega)$
		is any cut--off function satisfying $\eta \ge 0$ and $\eta = 1$ in $B_{\varepsilon}(y)$, 
		$\varepsilon > 0$.
	\item[(c)] $G(x,y) = G(y,x)$ for every $y \ne x$
	\item[(d)] $G(x, \cdot) \in L^1(\Omega)$ and 
	\[
		Tf(x) = \int_\Omega G(x,y) f(y) dy, \qquad \text{for all $f \in C^\infty_0(\Omega)$}
	\]
\end{enumerate}
\label{GreensFctnExists}
\end{thm}

Considering the low regularity expected for $f$ in \eqref{InHomoEqn} in the context
of $w = u_1 - u_2$, we extend the representation found in Theorem~\ref{GreensFctnExists}(d)
to more general functions $f$:



\begin{lemma} Let $D$ be a bounded, connected, Lipschitz domain in $\R^n$, 
let $G$ be the Dirichlet Green's function on $D$, and consider
\[
           f \in L^q(D) \qquad \text{with $q > n.$}
\]
Then the solution $u = Tf$ to \eqref{InHomoEqn} satisfies the representation
\[
	u(x) = \int_D G(x,y) f(y) dy \qquad \text{for all $x \in D$.}
\]
\label{GreensRep}
\end{lemma}

\begin{proof}
Fix $x \in D$.
Since the Green's function $G(x,\cdot)$ belongs to $W^{1,p}_{0}(D)$ for all $p \in [1, n/(n-1))$
(see \cite[Theorem 1.2.8]{K}), the map given by:
\[
         If := \int_D G(x,y) f(y) dy
\]
is a bounded continuous linear functional on $L^q(D)$ for all $q > n/2$ by H\"older's inequality.
By Calderon-Zygmund theory
(see \cite[Chapter 9]{GT}), it follows that the solution map 
$T: f \mapsto u,$ taking $f \in L^q(D)$ with
$n/2 < q < \infty$ to the solution  $u := Tf \in W^{2,q}_0(D)$ of
\begin{equation}\label{InHomoEqnAgain}
	\begin{cases}
		\Delta u = f & \text{in $D$}\\
		u = 0 & \text{on $\partial D$}
	\end{cases}
\end{equation}
is a bounded linear map.  

Further, since $W^{2,q}_0(D) \subset C^{1,\alpha}_0(\overline{D})$
when $q > n,$ it follows that the map $\tilde{T}: f \mapsto u(x)$ 
(composition of $T$ and pointwise evaluation at $x \in D$) is also continuous.  
Thus, we know by Theorem~\ref{GreensFctnExists}(d) that the maps $\tilde{T}$ and $I$ 
agree whenever $f \in C^\infty_0(D).$ Since $C^\infty_0(D)$ is dense in
$L^q(D)$ for all $n < q < \infty,$ we know that $I(f)$ and $\tilde{T}(f)$ 
must agree for all $f \in L^q(D),$ when $q > n.$
\end{proof}


For any parameter $\eta > 0$, we note that the restricted domain
\[
	D_{-\eta} := D \setminus \mathcal{N}_{\eta}(\partial D) 
		= \{ x \in D : {\rm dist}(x, \partial D) \ge \eta \}
\]
is a compact subset of $D$. Thus, the following uniform bounds on the Green's
function follow from regularity and positivity of $G$ (away from the pole and away 
from the boundary $\partial D$).

\begin{proposition}
\label{GreenFctnBounds}
Fix $\delta > 0$ so that $D_{-\delta} \ne \emptyset$ and consider Green's function
$G(\bar y, \cdot)$ with singularity at $\bar y \in D_{-\delta}$, then:
\begin{itemize}
	\item[(i)] for $\eta > 0$, there is a constant 
		$\underline G = \underline G (n, D, \delta, \eta) > 0$ so that
		\[
			-G(x, \bar y) \ge \underline G \qquad 
				\text{for $x \in D_{- \eta}$.}
		\]
	\item[(ii)] there is a constant
		$\overline G = \overline G (n, D, \delta) > 0$ so that 
		\[
			-G(x, \bar y) \le \overline G \qquad 
				\text{for $x \in D \setminus B_{\delta}(\bar y)$.}
		\]
\end{itemize} 
\end{proposition}

\subsection{The Homogeneous Dirichlet Problem in $D$}

Turning to the situation in \eqref{ObsProb} when Laplacian data is fixed (i.e. assuming $g_1 = g_2$), 
the difference $w = u_1 - u_2$ can be written as the sum of a solution to
the an inhomogeneous Dirichlet \eqref{InHomoEqn} and a harmonic function $\Cartman$
satisfying a homogeneous Dirichlet problem of the form
\begin{equation}\label{HomoEqn}
	\begin{cases}
		\Delta \Cartman = 0 & \text{in $D$}\\
		\Cartman = \phi & \text{on $\partial D$.}
	\end{cases}
\end{equation}
To bound the function $\Cartman$ and access the boundary data $\phi = \psi_1 - \psi_2$, 
we utilize harmonic measures and properties of Poisson kernels in Lipschitz domains. 
The sensitive dependence of solutions to boundary value problems and the regularity
of the boundaries themselves has been an area of deep inquiry with contributions from many
mathematicians. Although many great references can be included in this context, we refer 
the reader to \cite{K} for a detailed development of the content necessary for our setting.

We first note that \eqref{HomoEqn} is solvable for Lipschitz $D$ and continuous $\phi$:

\begin{thm}\cite[Theorems 1.3.1, 1.3.2(3) and equation (1.3.6)]{Sh}: 
	Let $D$ be a bounded Lipschitz domain. Given any $\phi \in C(\partial D),$ there exists a 
	$\Cartman \in C(\overline D)$ satisfying \eqref{HomoEqn}. Moreover, for every 
	$y \in D$ there exists a function $K(y,\cdot) \in C^\alpha(\partial D)$, for some
	$0 < \alpha < 1$, so that $\Cartman$ satisfies the expression
	\[
		\Cartman(y) = \int_{\partial D} \phi(x) K(y,x) d\sigma(x).
	\]
\label{PoissonKernel}
\end{thm}

The function $K(y,\cdot)$ is the Poisson kernel on $D$, which can be defined in general as the 
Radon--Nikodym derivative of harmonic measure $\omega^y$ with respect to surface measure
$\sigma$ on $\partial D$. Other expressions for $K(y,\cdot)$ can also be found in
\cite[Corollaries 1.3.18 and 1.3.19]{K}, for instance. Moreover, by \cite[Theorem 1.3.17]{K}
and the definition of kernel function, we conclude that $K(y, x) > 0$ whenever $y \notin \partial D$.
Thus, by compactness of $D_{-\delta}$ and continuity of $K(y, \cdot)$ on $\partial D$, we 
derive the following bounds on $K$:

\begin{proposition}
\label{PoissonKernelBounds}
Fix $\delta > 0$ so that $D_{-\delta} \ne \emptyset.$ Then there exist positive constants
$\overline{K} = \overline{K}(n,D,\delta)$ and $\underline{K} = \underline{K}(n,D,\delta)$ 
so that
\begin{equation}
	\underline{K} \le K(\bar y, \cdot) \le \overline{K}, \qquad
		\text{for all $\bar y \in D_{-\delta}$.}
\end{equation}
\end{proposition}

\section{Measure Theoretic Changes to Contact Sets}\label{sec:Proof}

We now proceed with the proof of our main result, Theorem~\ref{MainResult}.
As a general overview, we first isolate cases where either the Laplacian or the boundary
data are fixed.
We prove results in each of these cases first, then we conclude the proof
of our main result by applying standard ordering principles on solutions to the obstacle problem.

\begin{lemma}[Linear control with Perturbed Boundary Data]
Take $u_i$ and $\bar v$ as in Theorem~\ref{MainResult} and assume that $g = g_1 = g_2$.
\begin{itemize}
	\item[(a)] (Linear Stability) Suppose $\bar y \in D \cap \Lambda(\bar v)$ with 
		dist$(\bar y, \partial D) \ge \delta > 0,$ and choose $\eta > 0$. Then
		\[
			|(\Lambda(u_1) \; \Delta \; \Lambda(u_2)) \cap D_{-\eta}| \le 
				\left( \frac{\overline K(n, D, \delta)}{\lambda \, \underline{G}(n, D, \eta, \delta)} \right)
				\| \psi_1 - \psi_2 \|_{L^1(\partial D)}.
		\]
	\item[(b)] (Linear Non--Degeneracy) Suppose $\psi_1 \ge \psi_2$ on $\partial D$ and 
		$B_\delta(\bar y) \subset \Lambda(u_1)$ for some $\delta > 0$. Then
		\[
			|\Lambda(u_1) \; \Delta \; \Lambda(u_2)| \ge
				\left( \frac{\underline{K}(n, D, \delta)}{\mu \, \overline G(n, D, \delta)}\right)
				\| \psi_1 - \psi_2 \|_{L^1(\partial D)}.
		\]
\end{itemize}
\label{BoundaryLemma}
\end{lemma}

\begin{proof}
{\bf (a)} To prove linear stability, we define $\underline v := \OS{g}{\min(\psi_1, \psi_2)}$ and
note that $\underline v \le u_i \le \bar v$ holds in $D$, $i = 1, 2$. Therefore, we have 
\[
	\Lambda(u_1) \; \Delta \; \Lambda(u_2) \subset 
		\Lambda(\underline v) \; \Delta \; \Lambda(\bar v) =: \mathcal{L},
\]
and it suffices to prove the desired bound for $\mathcal{L} \cap D_{-\eta}.$

Define the auxiliary function $\Cartman$ solving
\begin{equation}\label{CartmanEqn}
	\begin{cases}
		\Delta \Cartman = 0 &\text{in $D$}\\
		\Cartman = |\psi_1 - \psi_2| &\text{on $\partial D$},
	\end{cases}
\end{equation}
and define $h := \bar v - \underline v - \Cartman$. Note that $h$ verifies
$h(\bar y) = - \Cartman (\bar y)$ and
\begin{equation}\label{hEqn}
	\begin{cases}
		\Delta h = \chisub{\mathcal{L}} g &\text{in $D$}\\
		h = 0 &\text{on $\partial D$}.
	\end{cases}
\end{equation}

Since $\bar y \in D_{-\delta}$ and $\Cartman$ solves \eqref{CartmanEqn},
we apply Theorem~\ref{PoissonKernel} and Proposition~\ref{PoissonKernelBounds}
to conclude the existence of $\overline K > 0$ such that 
\begin{equation}
\label{FirstBounds}
	\overline{K} \, \| \psi_1 - \psi_2 \|_{L^1(\partial D)}
  		\ge \int_{\partial D} |\psi_1 - \psi_2| \, K(\bar y, \cdot) \, d\sigma = \Cartman(\bar{y}). 
\end{equation}

By Proposition~\ref{GreenFctnBounds}(a) there exists $\underline G > 0$ 
so that $-G(\bar y, x) \ge \underline G$ for all $x \in D_{-\eta}$. 
Further, by $g \in L^\infty(D)$ and $\mathcal{L}$ measurable, we conclude that 
$\chisub{\mathcal{L}} \in L^{q}(D)$ for any $q > n$, so combining 
\eqref{hEqn} and Lemma~\ref{GreensRep}, we compute
\begin{align*}
 	\Cartman(\bar{y}) &= - h(\bar{y}) \\
 					&= - \int_{\mathcal{L}} g(x) \, G(\bar y, x) \, dx \\
  					&\ge \lambda \int_{\mathcal{L}} - G(\bar y, x) \, dx\\
  					&\ge \lambda \int_{\mathcal{L} \; \cap \; D_{-\eta}} -G(\bar{y}, x) \, dx \\
    					&\ge \lambda \underline{G} \; |\mathcal{L} \cap  D_{-\eta}| \, .
\end{align*}
Together with \eqref{FirstBounds}, this completes the proof of (a).\\

\medskip

{\bf (b)} To prove linear non--degeneracy, we use the same tools constructed in
the proof of (a), noting that $\psi_1 \ge \psi_2$ implies that $\bar v = u_1$, $\underline v = u_2$,
and $\mathcal{L} = \Lambda(u_1) \; \Delta \; \Lambda(u_2)$ in this case.
Also note that $h = u_1 - u_2 - \Cartman$ satisfies \eqref{hEqn}. 

By assumption that $B_{\delta}(\bar y) \subset \Lambda(u_1)$, we have
$\Lambda(u_1) \; \Delta \; \Lambda(u_2) \subset D \setminus B_{\delta}(\bar y)$
and so it follows from Theorem~\ref{GreenFctnBounds}(b) that there exists $\overline G > 0$
so that $-G(\bar y, x) \le \overline G$ for all $x \in \Lambda(u_1) \; \Delta \; \Lambda(u_2)$.
Therefore, employing Theorem~\ref{PoissonKernel}, Proposition~\ref{PoissonKernelBounds},
and Lemma~\ref{GreensRep}, we compute
\begin{align*}
	\underline{K} \, \|\psi_1 - \psi_2 \|_{L^1(\partial D)} 
		&\le \int_{\partial D} |\psi_1 - \psi_2| \, K(\bar y, \cdot) \, d\sigma\\ 
		&= \Cartman(\bar y) = -h(\bar y) \\
		&= -\int_{\Lambda(u_1) \Delta \Lambda(u_2)} g(x) \, G(\bar y, x) \, dx \\
		&\le \mu \int_{\Lambda(u_1) \Delta \Lambda(u_2)} - G(\bar y, x) \, dx\\ 
		&\le \mu \,  \overline G \ | \Lambda(u_1) \; \Delta \; \Lambda(u_2)|,
\end{align*}
which completes the proof of (b).
\end{proof}

\begin{lemma}[Linear control with Perturbed Right Hand Side]
Take $u_i$ and $\bar v$ as in Theorem~\ref{MainResult} and assume that $\psi = \psi_1 = \psi_2$.
Further, assume $\delta, \eta > 0$ are fixed and $B_\delta(\bar y) \subset \Lambda(\bar v)$
for some $\bar y \in D$.
\begin{itemize}
	\item[(a)] (Linear Stability) We have
		\[
			|(\Lambda(u_1) \; \Delta \; \Lambda(u_2)) \cap D_{-\eta}| \le 
				\left( \frac{\overline G(n, D, \delta)}{\lambda \, \underline{G}(n, D, \eta, \delta)} \right)
				\| g_1 - g_2 \|_{L^1(\Omega(\bar v))}.
		\]
	\item[(b)] (Linear Non--Degeneracy) Suppose $g_1 \le g_2$ in $D$ and 
		$B_\delta(\bar y) \subset \Lambda(u_1)$ for some $\delta > 0$. Then
		\[
			|\Lambda(u_1) \; \Delta \; \Lambda(u_2)| \ge
				\left( \frac{\underline{G}(n, D, \eta, \delta)}{\mu \, \overline G(n, D, \delta)}\right)
				\| g_1 - g_2 \|_{L^1(\Omega(\bar v) \, \cap \, D_{-\eta})}.
		\]
\end{itemize}
\label{LaplacianLemma}
\end{lemma}

\begin{proof}
{\bf (a)} For linear stability, we again define $\underline v := \OS{\max(g_1, g_2)}{\psi}$,
so that $\underline v \le u_i \le \bar v$ again holds, thus it suffices to prove the result for 
\[
	\mathcal{L} := \Lambda(\underline v) \; \Delta \; \Lambda(\bar v) \supset
		\Lambda(u_1) \; \Delta \; \Lambda(u_2).
\]

We define the auxiliary function $\Phi$ solving
\begin{equation*}\label{PhiEqn}
\begin{cases}
	\Delta \Phi = \chisub{\Omega(\bar v)} |g_1 - g_2| & \text{in $D$}\\
	\Phi = 0 & \text{on $\partial D$}
\end{cases}
\end{equation*}
and define $h := \overline v - \underline v + \Phi$. 
It follows that $h(\bar y) = \Phi(\bar y)$ and
\begin{equation}\label{hOtherEqn}
\begin{cases}
	\Delta h = \chisub{\mathcal{L}} \max(g_1, g_2) &\text{in $D$}\\
	h = 0 & \text{on $\partial D$.}
\end{cases}
\end{equation}
Note that we have $\chisub{\mathcal{L}} \max (g_1, g_2) \in L^{q}(D)$ for
any $q > n$, and the assumption on $\bar y$ ensures 
$\Omega(\bar v) \subset D \setminus B_{\delta}(\bar y)$.
Thus, we apply Proposition~\ref{GreenFctnBounds}, Lemma~\ref{GreensRep}, 
and $g \ge \lambda$ in $D$ to compute
\begin{align*}
  \overline{G} \, \| g_1 - g_2 \|_{L^1(\Omega(\bar v))}
  		&\ge - \int_{\Omega(\bar v)} |g_1(x) - g_2(x)| \, G(\bar y, x) \, dx \\
  		&= -\Phi(\bar{y}) = - h(\bar{y})\\ 
  		&= -\int_{\mathcal{L}} \max(g_1(x), g_2(x)) \, G(\bar y, x) \, dx \\
  		&\ge \lambda \int_{\mathcal{L}} -G(\bar{y}, x) \, dx \\
  		&\ge \lambda \int_{\mathcal{L} \; \cap \; D_{-\eta}} - G(\bar{y}, x) \, dx \\
    		&\ge \lambda \, \underline{G} \ |\mathcal{L} \cap  D_{-\eta}| \, ,
\end{align*}
which completes the proof of (a).\\

\medskip

{\bf (b)} For linear non--degeneracy, we again use the same tools constructed in
the proof of (a). With $g_1 \le g_2,$ we have $\bar v = u_1$, $\underline v = u_2$,
and $\mathcal{L} = \Lambda(u_1) \; \Delta \; \Lambda(u_2)$ in this case.
Note that $h = u_1 - u_2 + \Phi$ satisfies \eqref{hOtherEqn} where
$\max(g_1, g_2) = g_2$ in this case. 
Thus, applying Proposition~\ref{GreenFctnBounds}, Lemma~\ref{GreensRep} 
and $g \le \mu$ in $D$, we have
\begin{align*}
	\underline{G} \, \|g_1 - g_2 \|_{L^1(\Omega(\bar v) \cap D_{-\eta})} 
		&\le -\int_{\Omega(\bar v) \cap D_{-\eta}} (g_2(x) - g_1(x)) \, G(x, \bar y) \, dx \\
		&\le -\int_{\Omega(\bar v) \cap D_{-\eta}} (g_2(x) - g_1(x)) \, G(x, \bar y) \, dx \\
		&= -\Phi(\bar y) = -h(\bar y) \\
		&= -\int_{\Lambda(u_1) \Delta \Lambda(u_2)} g_2(x) \, G(x, \bar y) \, dx \\
		&\le \mu \int_{\Lambda(u_1) \Delta \Lambda(u_2)} -G(x, \bar y) \, dx \\
		&\le \mu \,  \overline G \ | \Lambda(u_1) \; \Delta \; \Lambda(u_2)|,
\end{align*}
which completes the proof of (b).
\end{proof}

\subsection{Proof of Theorem~\ref{MainResult}}
We conclude the note with a quick comment on bringing together the results from 
the preceding Lemmata to prove Theorem~\ref{MainResult}.

\begin{proof}[Proof of Theorem~\ref{MainResult}]
{\bf (a)} Regarding linear stability, we recall $\bar v$ as defined in the statement
of the theorem and further define 
\begin{align*}
	\underline v &= \OS{\max (g_1, g_2)}{\min (\psi_1, \psi_2)}, \qquad \text{and}\\
	w &= \OS{\max (g_1, g_2)}{\max (\psi_1, \psi_2)}.
\end{align*}
Notice that we can apply Lemma~\ref{BoundaryLemma}(a) to the set difference
$\Lambda(w) \; \Delta \; \Lambda(\underline v)$, while Lemma~\ref{LaplacianLemma}(a) applies to 
$\Lambda(\bar v) \; \Delta \; \Lambda(w)$. The proof of part (a) of Theorem~\ref{MainResult}
thus follows from these Lemmata and the simple bound
\begin{align*}
	|(\Lambda(u_1) \; \Delta \; \Lambda(u_2)) \cap D_{-\eta}|
		&\le |(\Lambda(\bar v) \; \Delta \; \Lambda(\underline v) \cap D_{-\eta}| \\
		&\le |(\Lambda(\bar v) \; \Delta \; \Lambda(w)) \cap D_{-\eta}| + 
		|(\Lambda(w) \; \Delta \; \Lambda(\underline v)) \cap D_{-\eta}|.
\end{align*}

{\bf (b)} Proving non--degeneracy follows in a similar manner, where here the 
function $w$ satisfies
\[
	w = \OS{g_2}{\psi_1},
\]
due to the monotonicity assumptions on $\psi_i$ and $g_i$.
The result now follows by applying Lemma~\ref{BoundaryLemma}(b) to 
$\Lambda(w) \; \Delta \; \Lambda(u_2)$, applying Lemma~\ref{LaplacianLemma}(b)
to $\Lambda(u_1) \; \Delta \; \Lambda(w)$, and noting that these sets 
form a disjoint decomposition of $\Lambda(u_1) \; \Delta \; \Lambda(u_2)$ in this case.
\end{proof}

\end{document}